\documentclass[12pt]{article}

\usepackage{comment}

\usepackage{amsmath}
\usepackage{amsthm}
\usepackage{amsfonts}
\usepackage{amssymb}
\usepackage{mathrsfs}
\usepackage{amstext}
\usepackage{amsopn}
\usepackage{graphicx}
\usepackage{subfig}
\usepackage{color}
\usepackage{rotating}
\usepackage{multirow}
\usepackage{framed}
\usepackage{float}
\usepackage{authblk}
\usepackage{cleveref}
\usepackage{amsmath}
\usepackage{amsthm}
\usepackage{amsfonts}
\usepackage{amssymb}
\usepackage{mathrsfs}
\usepackage{amstext}
\usepackage{amsopn}
\usepackage{graphicx}
\usepackage{subfig}
\usepackage{color}
\usepackage{rotating}
\usepackage{multirow}

\usepackage{float}
\usepackage[pagewise]{lineno}%\linenumbers

\usepackage{fancyhdr}
\def\N{\mathbb N}
\def\Z{\mathbb Z}

\def\C{\mathbb C}

\theoremstyle{definition}
	\newtheorem{definition}{Definition}[section]
	\newtheorem{example}{Example}

	\newtheorem{lemma}{Lemma}
	
	\newtheorem{theorem}{Theorem}
\newcounter{tmp}
	\newtheorem{corollary}{Corollary}
	\newtheorem{remark}{Remark}
\theoremstyle{remark}
\def \Cheng  {Cheng-Minkowycz} % also BBT = Belhachmi, Brighi, Taous,

\author[1,2,3]{Robert Conte}%\thanks{robert.conte@cea.fr}}
\author[3]{Tuen-Wai Ng}%\thanks{ntw@maths.hku.hk}}
\author[1]{Chengfa Wu*}%\thanks{Corresponding author: C. F. Wu(Chengfa Wu)}}

\affil[1]{Institute for Advanced Study, Shenzhen University, Shenzhen, PR China}
\affil[2]{Universit\'e Paris-Saclay, ENS Paris-Saclay, CNRS, Centre Borelli, F-91190, Gif-sur-Yvette, France}
\affil[3]{Department of Mathematics,
The University of Hong Kong,
Pokfulam Road,
Hong Kong}

\title{Closed-form meromorphic solutions of some third order boundary layer ordinary differential equations}
\date{}%May 29, 2019}                                        % PLEASE UPDATE

\begin{document}

\maketitle

\begin{figure}[b]
\rule[-2.5truemm]{5cm}{0.1truemm}\\[2mm]
{\footnotesize
\par
* Corresponding author: C. F. Wu (Chengfa Wu)
}

\end{figure}

\abstract{
We introduce a general third order non-linear autonomous ODE which covers many ODEs coming from boundary layer problems,
like the Falkner-Skan equation and the \Cheng\ equation. Using Wiman-Valiron theory and complex analytic methods recently developed, for the generic cases, it is shown that all their meromorphic solutions must be rational, or rational in one exponential, and then we find all of them explicitly. For a few non-generic cases, some solutions, which are meromorphic or singlevalued, are also obtained. Our results also explain why it is so difficult to obtain new closed-form solutions of the Falkner-Skan equation.
}

%\noindent
%\textit{Keywords}:
%{\bf Key words.} Meromorphic solutions,
%Complex differential equations,
%Painlev\'{e} analysis,
%Nevannlina theory,
%Wiman-Valiron theory.
%rational,

%\noindent
% {\bf AMS subject classifications. 35Q85, 35A24, 35C09}

%\tableofcontents

%\vfill\eject
% ============================================================================

\section{Introduction}

In this paper, we use complex analytic methods to find  meromorphic solutions of an equation which covers many ODEs stemming from boundary layer problems,
such as the   Falkner-Skan equation \cite{Falkner1931Skan} and the \Cheng\ equation \cite{Cheng1977Minkowycz}, namely
the following third order non-linear ODE
 \begin{eqnarray}
u'''(z) - B u(z) u''(z) - A {u'(z)}^2 +\alpha u''(z)+ \beta u'(z) + \gamma u(z) +\delta=0,\
%A B \not=0,\
A, B,  \alpha, \beta, \gamma,  \delta \in \mathbb{R}. %RC NOT C, fluid mechanics people won't like it
\label{third-order-equation}
%\label{eqODE3General}
\end{eqnarray}

Boundary layer equations    \cite{DrazinRiley,Schlichting2017Gersten}
%have been extensively studied due to their significance in  boundary layer theory \cite{} which
play    quite an important role in fluid mechanics and engineering. They arise as simplifications  of the Navier-Stokes equations for fluids  near solid boundaries with sufficiently high Reynolds numbers, and have many applications in industry especially  the design of airships and airplanes \cite{Pandya2014HuangEspitiaUranga}. Various numerical methods \cite{Keller1978,Skote2002Henningson} have been adopted to study boundary layer equations, such as the Crank-Nicolson scheme,   the Box scheme and their variations (see \cite{Keller1978,Riccardo1998} and  references therein). Despite that, it is generally
very difficult to construct exact solutions of the boundary layer equations. Fortunately, when we  focus on the so-called similarity solutions \cite{Hartman1972}, which have the property that  the velocity profiles at different positions around the solid boundaries  are the same apart from a scaling, the boundary layer (partial differential) equations can be  simplified significantly and in certain cases they may be transformed into ordinary differential equations. In particular, if we consider the steady two-dimensional  flow past a wedge of angle $\pi \lambda $, then through  similarity   transformations, the boundary layer  equations reduce to an ordinary differential equation subject to some boundary conditions
\begin{eqnarray}
f'''(\eta) + f (\eta) f''(\eta) + \lambda (1-{f'(\eta)}^2)=0, \label{eqFalkner-Skan-original}
 \\
f(0)= f'(0)=0,\ f'(+\infty)=1, \label{boundary-condition-Falkner-Skan-original} %(\lambda >0???),
\end{eqnarray}
which is called the Falkner-Skan equation \cite{Falkner1931Skan,Llibre2013Valls}. Here,   $\eta $ is the similarity variable %(what is $\lambda$???)
and $f (\eta)$ is the similarity stream function. For $\lambda = 0$,  the boundary layer is around a flat plate   parallel to a unidirectional flow with constant velocity, and the equation \eqref{eqFalkner-Skan-original} reduces to the Blasius equation \cite{Blasius1908,Boyd2008,Iacono2015Boyd}.
The famous Falkner-Skan equation \eqref{eqFalkner-Skan-original} has been studied intensively by many people (including mathematicians like Weyl \cite{Weyl1942} as well as Swinnerton-Dyer and his collabrators \cite{Swinnerton-Dyer1995Sparrow,Sparrow2002Swinnerton-Dyer}) since 1931. One now knows that there exists a critical value $\lambda_0 = -0.1988 \pm 0.0005$ defining three regimes
of solutions of \eqref{eqFalkner-Skan-original} obeying \eqref{boundary-condition-Falkner-Skan-original}
\cite{Falkner1931Skan,Hartree1937,Stewartson1954}:
\begin{itemize}
  \item [(i)] no  solution  for $\lambda < \lambda_0$;
  \item [(ii)] two solutions for $\lambda_0 < \lambda < 0$;
  \item [(iii)] one solution  for $0 < \lambda$.
\end{itemize}
The rigorous treatment of the solvability of the boundary value problem \eqref{eqFalkner-Skan-original} and  \eqref{boundary-condition-Falkner-Skan-original} was given by Weyl \cite{Weyl1942}. He proved that, for any $\lambda \geq 0$, there exists a solution $f(\eta)$ of \eqref{eqFalkner-Skan-original} and  \eqref{boundary-condition-Falkner-Skan-original} with the property that $f'$ and $f''$ are respectively increasing and decreasing  on $(0, \infty)$. A direct proof of this result was given by Coppel \cite{Coppel1960} whose argument   shows that the equation \eqref{eqFalkner-Skan-original} is solvable as well for the boundary conditions
\begin{eqnarray} \label{boundary-condition-Falkner-Skan-general}
f(0)= a, f'(0)=b,\ f'(+\infty)=1, \quad a<0, b<0. %(\lambda >0???), \quad
\end{eqnarray}
%where $a,b$ are nonnegative constants.
Apart from that, intensive study has been taken on  the properties of  solutions of \eqref{eqFalkner-Skan-original} with various  initial or boundary conditions, ranging from the existence of periodic solutions \cite{Hastings1987Troy} and multiple solutions \cite{Riley1989Weidman}, oscillation of solutions \cite{Hastings1987Troy,Hastings1988Troy} to the solution dynamics and bifurcations \cite{Swinnerton-Dyer1995Sparrow,Sparrow2002Swinnerton-Dyer}.  Although the existence of solutions of \eqref{eqFalkner-Skan-original} has been verified, very few exact solutions  have   been derived  except for   $\lambda = -1$ whose solutions can be expressed in terms of the parabolic cylinder functions  \cite{Coppel1961} or the confluent hypergeometric functions \cite{Yang1975Chien}. Theorem  \ref{Main Theorem} will explain why it is so difficult to find exact meromorphic solutions for the Falkner-Skan equation \eqref{eqFalkner-Skan-original} since 1931 and Remark \ref{remark_Chazy} will give an explicit infinite set of rational numbers for the coefficient $\lambda$ that one should consider in order to find new exact meromorphic solutions of \eqref{eqFalkner-Skan-original}.

 The similarity transformations are also applicable to   boundary layer problems in other physical contexts, and the resulting equations may still be ordinary differential equations. For example,  Cheng and Minkowycz \cite{Cheng1977Minkowycz} showed that the governing partial differential  equations of free convective boundary layer flow over a vertical flat plate  immersed in a porous medium can be reduced to
 \begin{eqnarray}
 \begin{cases} \label{Cheng-Minkowycz}
f''' (\eta) + \dfrac{a+1}{2} f(\eta) f''(\eta) - a {f'(\eta)}^2=0,\ a(a+1) \not=0,
\\
f(0)=0,\ f'(0)=1,\ f'(+\infty)=0,
\end{cases}
 \end{eqnarray}
 which is now called the Cheng-Minkowycz equation \cite{Na1996Pop}.

 The goal of this paper is to study meromorphic solutions of the third-order differential equation \eqref{third-order-equation} which covers both the Falkner-Skan equation \eqref{eqFalkner-Skan-original} and the Cheng-Minkowycz equation \eqref{Cheng-Minkowycz}.  It is known
%was established by Briot and Bouquet
\cite{Briot1875Bouquet} that if the first order autonomous ODE
\begin{equation} %\label{}
P(w,w') = 0,
\end{equation}
where $P$ is a polynomial, has a general solution   single-valued around all its   singularities which depend on the initial conditions (the so-called Painlev\'{e} property \cite{Conte1999}), then this solution must belong to the {\it class $W$}, which consists of  elliptic functions and their degeneracies  (rational functions of one exponential $\exp{ (kz)}, k \in \C$  and  rational functions of $z$). Further, if the second order ODE
\begin{equation}  %\label{}
  w''= F(w',w,x),
\end{equation}
with $F$ rational in $w' $ and $ w$, analytic in $x$, possesses the   Painlev\'{e} property, and its general solution has no fixed singularities, then this solution is meromorphic (unless otherwise specified, meromorphic functions refer to functions   defined on the whole complex plane $\C$ without singularities other than poles)  \cite{Painleve1900}. Therefore, meromorphic functions  are   the natural building blocks to construct particular solutions of higher-order ODEs. It is for this reason that we will study meromorphic solutions of the equation \eqref{third-order-equation}, and unless otherwise specified, the independent variable $z$ is assumed to be complex.

Our goal is \textit{not} to determine which equations \eqref{third-order-equation}
possess the Painlev\'e property,
a question fully solved by Chazy  \cite{Chazy1911} and Cosgrove \cite{Cosgrove2000},
but only to determine all their solutions in a particular class,
independently of their Painlev\'e property.

In terms of the construction of meromorphic solutions, % of autonomous algebraic ODEs,
Eremenko \cite{Eremenko2006} has proved that  there exists a class of autonomous algebraic ODEs such that all their meromorphic solutions (if they exist)   belong to the class $W$, and then  all of them could be explicitly obtained using either the  sub-equation method \cite{Conte2010Ng}   or the Hermite decomposition \cite{Demina2011Kudryashov,Hermite-sum-zeta}.

This method has been successfully applied to many differential equations \cite{Ng2019Wu,Yuan2015XiongLinWu}, like the Swift-Hohenberg equation \cite{Conte2012NgWong}, whose homoclinic solutions have been intensively studied in \cite{Santra2009Wei}. However, it turns out that the arguments based on Nevanlinna theory used in \cite{Conte2012NgWong,Eremenko2006} cannot be applied to the equation \eqref{third-order-equation} because it has two terms with top degree (see Theorem \ref{Eremenko's theorem} for the definition), and hence we have to make use of generalized Wiman-Valiron theory \cite{Bergweiler2008RipponStallard} instead, which is the novelty of this paper.

The structure of this paper is as follows. In Section \ref{Main Results}, the main results are presented. For generic cases (see Definition \ref{generic cases}), all meromorphic solutions of \eqref{third-order-equation}
 are shown to belong to the class $W$ and then explicitly derived. This conclusion does not hold for non-generic cases and several examples are provided. It is also shown that if one wants to obtain other new meromorphic solutions of the Falkner-Skan equation \eqref{eqFalkner-Skan-original}, then one has to consider some specific $\lambda$ such that $\lambda = 1- 1/r, r \in \Z \backslash \{0\}$ or   $\lambda =2+6j/[(j-1)(j-6)], j \in \N\cup \{0\} \backslash \{1,2,3,6\}$ (see Remark \ref{Application to FS eq}). A few non-generic cases that pass the Painlev\'{e} test are discussed in Section \ref{Special Cases with   General Solutions}. In Section \ref{Generalized Wiman-Valiron Theory}, we recall the generalized Wiman-Valiron theory for meromorphic functions with direct tracts   \cite{Bergweiler2008RipponStallard}.  Section \ref{Proof of Theorem} is devoted to the proof of the main results, where   Eremenko's method is  generalized. The main tools applied in the proof are the generalized Wiman-Valiron theory  and Painlev\'{e} analysis. We summarize the main results of this paper in Section \ref{conclusion}. Finally, a brief summary on Painlev\'{e} test is provided in Appendix A.

\section{Main Results} \label{Main Results}
We first recall Eremenko's result  on the classification  of meromorphic solutions   to a class of autonomous algebraic ODEs. % whose (if they exist)   belong to the class $W$.

 	\begingroup
	\setcounter{tmp}{\value{theorem}}
	\renewcommand\thetheorem{\Alph{theorem}}
	\begin{theorem} {\cite{Eremenko2006}} \label{Eremenko's theorem}
 If an  autonomous  algebraic ODE
 \begin{equation} \label{higher order algebraic ordinary differential equation}
 \displaystyle \sum_{\lambda \in I}  w^{i_0} (w')^{i_1} \cdots (w^{(n)})^{i_n} = 0,
 \end{equation}
where $I$ consists of finite multi-indices of the form $\lambda = (i_0,i_1 \cdots, i_n)$, $i_k \in \N$,
  satisfies
  \begin{itemize}
\item [1)] there is only one top degree term (the degree  of  each term in \eqref{higher order algebraic ordinary differential equation} is defined as
$|\lambda|= i_0 + i_1+ \cdots + i_n$),

\item [2)] there is no nonnegative integer Fuchs indices (see Appendix A for  the definition),
\end{itemize}
then all its meromorphic solutions belong to the {\it class W}.
	\end{theorem}
	\endgroup

\begin{remark} \label{two terms with top degree}
Since equation \eqref{third-order-equation} has two terms with top degree, Theorem \ref{Eremenko's theorem} cannot be applied to it.
\end{remark}

When applied to equation \eqref{third-order-equation}, this theorem yields the following results.
\begin{enumerate}
\item
For $A B \not=0$, there exist two top degree terms, Theorem \ref{Eremenko's theorem} cannot be applied.

\item
For $A = B = 0$, the equation \eqref{third-order-equation} is linear with constant coefficients.

\item
For $A = 0, B \not = 0$, the equation \eqref{third-order-equation} obeys conditions i) and ii) of Theorem \ref{Eremenko's theorem},
hence all its meromorphic solutions belong to the {\it class W}.

\item

For $A \not = 0, B  = 0$, by Wiman-Valiron
theory \cite{Laine1993}, we conclude that any entire solution of  \eqref{third-order-equation} must be a polynomial.   On the other hand, we can find that equation \eqref{third-order-equation} does not have any polynomial solution with degree greater than 2, while it is easy to construct its polynomial solutions whose degrees do not exceed 2. Therefore, it remains to consider     meromorphic solutions $u$ of  \eqref{third-order-equation} with at least one pole $z_0$ in $\C$. By direct computation, we find that $u$ has the following Laurent series expansion around $z_0$
\begin{eqnarray} \label{Laurent series expansion-B=0}
u = \sum_{n=0}^{+\infty} u_n(z-z_0)^{n-1}, \quad u_0   =  -\dfrac{6}{A}.
\end{eqnarray}
The Fuchs indices  of \eqref{Laurent series expansion-B=0} are $-1,1,6$
(corresponding to the arbitrary coefficients $z_0,u_1,u_6$)
and the  conditions for \eqref{Laurent series expansion-B=0} to exist  are $\alpha = \gamma = 0$. In this case, equation \eqref{third-order-equation} becomes
\begin{eqnarray}
%RC The writing lhs=degree two, rhs=degrees one and zero (Nevanlinna)
% is not the one of fluid mechanics.
%A u'(z)^2+B u(z) u''(z) = u^{(3)}(z) +\alpha  u''(z)+\beta  u'(z)+\gamma  u(z) + \delta,
f''' -    A {f'}^2 + \beta f' + \delta=0,\
%A B \not=0,\
\label{third-order-equation-B=0}
%\label{eqODE3General}
\end{eqnarray}
whose general solution is
\cite{Magyari}
\begin{eqnarray}
  f(z) = \dfrac{\beta }{2 A} z  - \dfrac{6 }{A} \zeta(z-z_0;g_2,g_3) + c_1 , %\, c_1\in\C,
\end{eqnarray}
where $\zeta(z;g_2,g_3)$ is the Weierstrass zeta function  with
\begin{eqnarray} \label{first-integral-B=0}
\begin{cases}
  g_2 = -\dfrac{A \delta }{3} - \dfrac{\beta ^2}{12},
   \\
   g_3 = \dfrac{A^2 c}{108}-\dfrac{A \beta  \delta }{36}-\dfrac{\beta ^3}{216},
   \end{cases}
\end{eqnarray}
and $z_0,c,c_1 \in \C$ are arbitrary.

\end{enumerate}

In the rest of this paper, we will focus on the ``generic'' case of equation \eqref{third-order-equation} that is defined as follows.
\begin{definition} \label{generic cases}
The ``generic'' case of equation \eqref{third-order-equation} is defined to be
 any case for which $A, B \in \C$ obey  satisfy all of the following three conditions
\begin{align}
   \text{i) } & AB( A+B) \not = 0; \label{generic-1}
   \\
   \text{ii) } &  \dfrac{B}{A+B} \notin \Z \backslash \{0\}; \label{generic-2}
    \\
   \text{iii) } &    \dfrac{7 A+8 B \pm \sqrt{25 A^2+16 A B-32 B^2}}{2 (A+2 B)} \not \in \N\cup \{0\}.  \hspace{11 em} \label{generic-3}
\end{align}
\end{definition}
%}
The main results are as follows.

\begingroup
	\setcounter{tmp}{\value{theorem}}% store current value of theorem counter
	\renewcommand\thetheorem{1}
\begin{theorem} \label{Main Theorem}
For any generic case, any meromorphic solution  of the equation \eqref{third-order-equation}
\begin{eqnarray*}
A u'(z)^2+B u(z) u''(z) = u'''(z) +\alpha  u''(z)+\beta  u'(z)+\gamma  u(z) + \delta, %\eqno{(1)}
%\label{third-order-ODE}
\end{eqnarray*}
where $AB \not = 0$, is   either
a rational function of the form
 \begin{eqnarray} \label{rational}
u(z) = \dfrac{u_0}{z-z_0} + P(z), %\quad u_0 \not = 0, - p\in \N.
\end{eqnarray}
 where $P(z)$ is a polynomial of degree at most two,
%\end{itemize}
or
%\begin{itemize}
%\item [ii)]
 a simply periodic function of the form
\begin{eqnarray} \label{rational in exp}
  u(z) &=& \dfrac{h_0}{e^{k z} - \zeta_0} + c_0, \quad  h_0, k, \zeta_0 \in \C^*, c_0 \in \C.
\end{eqnarray}
\end{theorem}
This theorem will be proven in Section \ref{Proof of Theorem}.

\begin{remark}
 We note that the solution \eqref{rational} is a polynomial when $u_0$ is zero.
\end{remark}

	\endgroup

 \begin{corollary} \label{explicit meromorphic solution}
 For  generic cases (see  Theorem \ref{Main Theorem}), all  nonconstant  meromorphic solutions of the ODE \eqref{third-order-equation} can be explicitly constructed and they are listed in Table \ref{table-meromorphic-solutions} (without loss of generality, $\alpha$ can be set to 0 by a translation and the arbitrary constant $z_0$ is absorbed in $z$ by representing $z-z_0$ as $z$).

 \end{corollary}

\begin{table}[h]%[h]%{h}%[htpb]

%\vspace{25em}
%\section*{Appendix}
%\centering
%\vspace{8em}
%\resizebox{0.77\textwidth}{!}{
%\begin{minipage}{\textwidth}
\newcommand{\tabincell}[2]{
%\scalebox{0.5}{
\begin{tabular}{@{}#1@{}}#2\end{tabular}}
\centering % used for centering table
\caption{All meromorphic solutions of the ODE  \eqref{third-order-equation} in the generic case, in which $\zeta_0 \in \C^*, b \in \C$ are arbitrary.}
\resizebox{\textwidth}{!}{
\begin{tabular}{   l  | l      } % centered columns (4 columns)
 \hline
 \hline %inserts double horizontal lines
\tabincell{c}{ Nonconstant meromorphic solutions \\ of the ODE  \eqref{third-order-equation}} & Constraints on the parameters
\\
\hline
% Rational in exp
$u(z) = \dfrac{h_0}{e^{k z} - \zeta_0} + c_0 $
%\end{eqnarray*}
& %the following conditions hold
$
\begin{cases}
 \gamma=  \delta = 0
 \\
 c_0 = -\dfrac{\beta +k^2}{B k}
 \\ h_0 = -\dfrac{2 \zeta_0  \left(\beta +k^2\right)}{B k}
 \\   \beta  (A+2 B)+k^2 (A-B)=0
 \end{cases}
$
\\
\hline
%Rational solutions
  $u(z) =  -\dfrac{\gamma  z^2}{4 (A-B)}+\dfrac{\beta  z}{2 (A-B)}-\dfrac{6}{z (A+2 B)}$ & %if one of the following conditions holds
%\begin{stack}
%1
%\\
%2
%\end{stack}
$
\begin{cases}
 \gamma=0
 \\ \beta ^2 (A-2 B)+4 \delta  (A-B)^2 = 0
 \end{cases}
$
or
$
\begin{cases}
 \gamma \not =0
 \\
  4A - B = 0
  \\ \beta = \delta = 0
 \end{cases}
$

% Quadratic polynomials:
\\
\hline  $u(z) = a z^2 + b z +c$ & %if one of the following conditions holds
$
\begin{cases}
 2 A + B \not = 0
 \\
  a  = \dfrac{\gamma }{2 (2 A+B)} \not=0
  \\ \beta = 0
  \\ c =  \dfrac{(2 A+B) \left(A b^2-\delta \right)}{2 A \gamma }
 \end{cases}
$
or
$
\begin{cases}
 2 A + B   = 0
 \\
  \beta   = \gamma  = 0
  \\ A \left(b^2-4 a c\right)-\delta =0
 \end{cases}
$
 \\
\hline
%Linear solutions:
$u(z) = a z +b  $ &
%   \begin{eqnarray}
$
\begin{cases}
 \gamma = 0
 \\
   A a^2-\beta  a-\delta = 0
 \end{cases}
$
 \\
 \hline
%\end{eqnarray}
\end{tabular}
}
\label{table-meromorphic-solutions} % is used to refer this table in the text
%\end{minipage}
%}
\end{table}

%{\color{red}
\begin{remark} \label{remark_Chazy}
Chazy \cite{Chazy1911} introduced 13 classes of third order ODEs in his classification of certain third-order differential equations with the Painlev\'{e} property. For comprehensive discussions on these equations, we also refer the readers to the reference \cite{Cosgrove2000} in which the solutions of Chazy equations IX and X are constructed in terms of hyperelliptic functions of genus 2. Nevertheless, equation \eqref{third-order-equation} is not covered by Chazy's classes except two particular cases: `$A=B, \beta = \gamma = 0$' and `$A=3, B=2, \alpha = \beta = \gamma = \delta = 0$', which correspond to Chazy equations II and III respectively and will be treated in Section \ref{Special Cases with   General Solutions}. This also implies that equation \eqref{third-order-equation} does not possess the Painlev\'{e} property except for specific choices of the coefficients.
\end{remark}
%}

\begin{remark} \label{Application to FS eq} According to  Theorem \ref{Main Theorem} and Corollary \ref{explicit meromorphic solution}, if one wants to construct other nongeneric meromorphic solutions of the Falkner-Skan equation  \eqref{eqFalkner-Skan-original}, then one has to consider some specific $\lambda$ such that $\lambda = 1- 1/r, r \in \Z \backslash \{0\}$ or at least one of $(8-7 \lambda \pm \sqrt{\lambda  (25 \lambda -16)-32})/(4-2 \lambda )$ is a nonnegative integer except $1, 2, 3  $ and $ 6$, i.e., $\lambda =2+6j/[(j-1)(j-6)], j \in \N\cup \{0\} \backslash \{1, 2, 3 ,  6\}$, such as  $ \lambda =   -2 \; (j=4), $ $ -11/2 \; (j=5), $ $ 9 \; (j=7)$ or $38/7 \; (j=8)$, because we will show in Section \ref{Special Cases with   General Solutions} that \eqref{eqFalkner-Skan-original} admits general solution for $\lambda = -1$, which corresponds to $j=2,3$. In particular, one has to focus on  $\lambda = 1- 1/r, r \in \Z \backslash \{0\}$ for physically meaningful meromorphic solutions as $\lambda \pi$ represents the angle of the wedge which requires  $-1 < \lambda <2$.  Otherwise, one has to concentrate on the solutions of \eqref{eqFalkner-Skan-original} with more complicated singularities other than poles.

\end{remark}

\begin{remark}
For non-generic cases, the conclusion of   Theorem \ref{Main Theorem}  does not hold and this is shown through the examples below. %{\color{red}
However, we do not know  what happens in general if one
of the conditions \eqref{generic-1}--\eqref{generic-3} is dropped.%}

\end{remark}
% ========================================================================

\begin{example}
Let  $\alpha=\beta=\gamma=\delta=0, A=- B = 1 $ (condition i) fails), then the  equation %\eqref{third-order-ode}   is the well-known    Chazy's equation
 \begin{eqnarray*}  %\label{Chazy Equation}
u'''     =  (u')^2 - u u'',
 \end{eqnarray*}
has an entire solution %RC
  \begin{eqnarray*}  %\label{Chazy Equation}
u(z) = e^{z-z_0} - 1, \quad z_0 \in \C.
 \end{eqnarray*}

\end{example}
% ========================================================================

\begin{example}
Let  $\alpha=\beta=\gamma=\delta=0, A=3, B = n A/(1-n)  $ , where $ n = -2$  (condition ii) fails), then the  equation \eqref{third-order-equation}   reduces to the well-known    Chazy equation III
 \begin{eqnarray}  \label{Chazy Equation}
u'''     = - 3 (u')^2 + 2 u u'',
 \end{eqnarray}
and its only globally meromorphic solution is the rational solution \cite[p.~335]{Chazy1911}%RC
  \begin{eqnarray*}  %\label{Chazy Equation}
u(z) = c/(z - z_0)^2 -6 / (z - z_0), \quad c,z_0 \in \C,
 \end{eqnarray*}
 with a double pole at $z_0$.

\end{example}

% ========================================================================
\begin{example}
Let  $\alpha=\beta=\gamma=0, A=  B = -2  $, then the general solution of
\begin{eqnarray*}  %\label{factoriable-third order equation}
u'''(z) =  -2 u'(z)^2  -2 u(z) u''(z) - \delta,
 \end{eqnarray*}
 which has Fuchs indices $j = -1, 2 , 3$ (condition iii) fails), is %meromorphic and is given by
given by
% \begin{eqnarray}  %\label{factoriable-third order equation}
$u = w' / w$, where $w$ satisfies
\begin{eqnarray*}
w'' -  (\delta z^2/2 + c_1 z + c _2)w = 0, \quad c_1, c_2 \in \C,
 \end{eqnarray*}
and hence $u$ can be expressed in terms of
  the parabolic cylinder functions ($\delta \not = 0$), the Airy functions ($\delta = 0, c_1 \not=0$), exponentials ($\delta =  c_1 =0, c_2 \not = 0$) %RC precise
or polynomials ($\delta =  c_1 = c_2  = 0$).

\end{example}

\section{Special Cases with   General Solutions} \label{Special Cases with   General Solutions}

Let us first require  equation  \eqref{third-order-equation} to pass the Painlev\'e test (see Appendix A) because it provides necessary conditions for the Painlev\'e property \cite{Conte1999}.
%, which guarantees  the general solution of  an ODE to be singlevalued around   movable singularities.
It is noted that the indicial equation (see Appendix A) of \eqref{third-order-equation} is
\begin{equation}
(i +1) \left(i^2 + \dfrac{  -7 A-8 B}{A+2 B} i +6\right) = 0
\end{equation}
and the diophantine equation $i_1 i_2=6$ has four solutions
\begin{equation*}
(i_1,i_2)= (-1,-6), \; (1,6), \;  (-2,-3), \;  (2,3)
\end{equation*}
%write (let a=11/3 in indicesBBT)/(i+1)/(i+6);
%write (let a=-1   in indicesBBT)/(i-1)/(i-6);
%write (let a=-1/3 in indicesBBT)/(i-2)/(i-3);
%write (let a=3    in indicesBBT)/(i+2)/(i+3);
for the respective relations $B/A = - 7/11, 0, -2/3, 1$. % $a=-1,-1/3,11/3,3$.
The case $(i_1,i_2)=(-1,-6)$ displays movable multivaluedness (double Fuchs index) while the case $(i_1,i_2)=(1,6)$, which corresponds to $B=0$, has been solved in Section 2. Hence, only two cases remain to be dealt with.

For $(i_1,i_2)=(2,3)$, i.e., $B=A$, further conditions for equation  \eqref{third-order-equation} to  pass the Painlev\'e test  are $\beta = \gamma = 0$, then equation  \eqref{third-order-equation}  reduces to a special case of Chazy equation II and  admits  the first integral
\begin{equation*}
 u' - \frac{A}{2} u^2 + \frac{\delta}{2} z^2 + k_1 z + k_0 = 0,
\end{equation*}
where $k_1,k_0$ are the integration constants. Therefore,  equation  \eqref{third-order-equation} has a singlevalued general solution given by
\begin{equation}
u=
\begin{cases}
-\dfrac{2}{A} \dfrac{w'_1(z-z_0)}{w_1(z-z_0)},\ \delta \not = 0,
\\
-\dfrac{2}{A} \dfrac{w'_2(z-z_0)}{w_2(z-z_0)},\ \delta = 0,  \ k_1  \not=0,
\\
\dfrac{2  k \tan [k (z-z_0)]}{A},\ k^2 = \dfrac{Ak_0}{2}, \delta = k_1=0,  \ k_0 \not=0,
\\
-\dfrac{2}{A} \dfrac{1}{z-z_0},\ \delta =  k_1=k_0=0,
\end{cases}
\end{equation}
where $z_0$ is arbitrary and $w_i,i=1,2$,  is the   general solution of
\begin{equation}
w'' +  \dfrac{1}{2} A  \left(-\dfrac{\delta  z^2}{2} + k_1 z+k_0\right)w = 0,
\end{equation}
which can be expressed in terms of the parabolic cylinder functions ($\delta \not = 0$) or the Airy functions ($\delta = 0, k_1 \not=0$). Hence,   the solution given in \cite{Coppel1960}  of the Falkner-Skan equation  \eqref{eqFalkner-Skan-original} with $\lambda = -1$    is rediscovered.
%corresponding to $\delta \not 0$ and $\delta  0$.

For $(i_1,i_2)=(-2,-3)$, i.e., $B=-2A/3$,  by using the perturbative Painlev\'{e} method, it has been proved in \cite{Conte1993FordyPickering}  that a necessary condition for equation  \eqref{third-order-equation} to  possess the Painlev\'e property is $\beta = \gamma =\delta = 0$. In this case, by a scaling in the independent variable $z$,  equation  \eqref{third-order-equation} reduces to the famous Chazy equation III  \eqref{Chazy Equation} whose general solution is singlevalued and can be expressed in terms of two solutions to the hypergeometric equation. In addition, the general solution of Chazy equation III  \eqref{Chazy Equation} has a movable natural boundary that is a circle    with center and radius depending on the initial conditions.

\begin{remark}
We have rediscovered  a particular solution \cite{Belhachmi2001BrighiTaous}
 \begin{equation*}
 f(z) = \sqrt{6} \tanh \left( \dfrac{t}{\sqrt{6}} \right)
 \end{equation*}
of the Cheng-Minkowycz equation \eqref{Cheng-Minkowycz} for $a = -1/3$, i.e., $A = B$. This case ($a = -1/3$) has certain special interest as it is related to a horizontal line source embedded in a porous medium. % and actually as we will show in Section \ref{Special Cases with   General Solutions} that it is general solution can be.

\end{remark}
%}

\begin{remark}
For generic values of the three constants of integration,
none of the general  solutions of \eqref{third-order-equation} is elliptic or degenerate   for the  cases $B = 0, B = -2 A /3$ and $ B = A$.
\end{remark}

\section{Generalized Wiman-Valiron Theory} \label{Generalized Wiman-Valiron Theory}

Wiman-Valiron theory \cite{Laine1993}  describes the asymptotic behaviours of  transcendental entire functions in certain discs
around points of maximum modulus and has found many applications in complex differential equations \cite{Laine1993}.  It has been generalized to certain class of meromorphic functions by Bergweiler, Rippon and Stallard \cite{Bergweiler2008RipponStallard}. In this section, we shall recall  some of  the main results in \cite{Bergweiler2008RipponStallard} as they play an essential role in the proof of   Theorem \ref{Main Theorem}.

%{\color{red}
\begin{definition} [\cite{Bergweiler2008RipponStallard}]
Suppose $\Omega$ is an unbounded domain in the complex plane $\C$ such that its boundary consists of piecewise
smooth curves and $\C \backslash \Omega$ is unbounded. Let $y$ be a complex-valued function whose domain of definition contains the closure $\overline{\Omega}$ of $\Omega$. Then $\Omega$ is called a {\it direct tract}
of $y$ if the following two conditions hold
\begin{itemize}
  \item $y$ is analytic in  $\Omega$ and continuous in $\overline{\Omega}$;
  \item there exists $R > 0$ such that $|y(z)| = R$
for $z \in   \partial \Omega $ while $|f(z)| > R$ for $z \in \Omega$.
\end{itemize}

  It  will be shown in Section  \ref{Proof of Theorem} that every transcendental meromorphic function with finitely many poles on the complex plane has a direct tract.

\end{definition}

Let $K>0$ and $y$ be a transcendental  meromorphic function on $\C$ which has a direct tract $\Omega$.  Let
%}
%Let $K>0$, $y$ be a  function meromorphic on $\C$   and $\Omega$ be a component of the set $\{z: |y(z)| > K\}$ in which $y$ is pole-free (so $\Omega$ is unbounded). Let
\begin{eqnarray*}
M(r)=M(r,  \Omega,y)= \max \{|y(z)|:|z|=r, z \in \Omega \},
\end{eqnarray*}
then the derivative
\begin{eqnarray*}
a(r) = d \log M(r) / d \log r =   \dfrac{r M'(r)}{M(r)}  %= M(r).
\end{eqnarray*}
exists for all $r > 0$ outside a countable set $\Lambda \subseteq (0, \infty)$.
By a theorem due to Fuchs \cite{Fuchs1981}, if $\infty$ is not a  pole of $y$, then we have
\begin{eqnarray} \label{Fuchs's result}
\dfrac{\log M(r)}{\log r} \rightarrow \infty \text{ and } a(r) \rightarrow \infty, \text{ as } r \rightarrow \infty, r \not \in \Lambda.
\end{eqnarray}
Assume $r_0 = \inf \{|z| : z \in \Omega \}$, and for any $r > r_0$, $z_r$ is chosen such that $|z_r| = r$ and $|y(z_r)| = M(r)$.

\begingroup
	\setcounter{tmp}{\value{theorem}}% store current value of theorem counter
	\setcounter{theorem}{1} %assign desired value to theorem counter
	\renewcommand\thetheorem{\Alph{theorem}}% locally redefine the representation of the theorem counter
\begin{theorem}[\cite{Bergweiler2008RipponStallard}] \label{Wiman-Valiron-mero1}
Let $y$  be a transcendental  meromorphic function on $\C$ with a direct tract $\Omega$. Then for every $\sigma > 1/2$, there exists a set $F \subset [1, \infty) $ of finite logarithmic measure such that for $r \in [r_0, \infty)\backslash F$, the disk
\begin{eqnarray*}
D_r =  \{z: |z - z_r| < r a^{-\sigma}(r)  \}
\end{eqnarray*}
is contained in $\Omega$. Moreover, we have
%\begin{itemize}
%\item [i)]
\begin{eqnarray*}
y^{(k)}(z) = \left( \dfrac{a(r)}{z} \right)^k \left( \dfrac{z}{z_r} \right)^{a(r)} y(z) (1 + o(1)), \quad  z \in D_r,
\end{eqnarray*}
as $r \rightarrow \infty, r \not \in F$.
	\endgroup

%\begin{theorem}
\begin{lemma}[\cite{Bergweiler2008RipponStallard}]  \label{Wiman-Valiron-mero2}
For any $\beta > 0$, we have
\begin{eqnarray*}
(a(r))^{\beta} = o(M(r)),
\end{eqnarray*}
as $r \rightarrow \infty $ outside a set of finite logarithmic measure.

\end{lemma}
\section{Proof of Theorem \ref{Main Theorem}} \label{Proof of Theorem}

\begin{lemma} \label{infinitely many poles}
If $A+B \not = 0$, then any transcendental meromorphic solution of equation \eqref{third-order-equation}
has infinitely many poles.

\end{lemma}

\begin{proof}
We prove by contradiction. Suppose $u$ has finitely many poles on $\C$, then there exists $L>0$ such that $u$ is analytic in $D = \{ z:  |z| \geq L \}$.
   Choose $R > \max_{|z| = L}|u(z)|$. Let $\Omega$ be a component of the set $\{z: |u(z)| > R, |z| > L\}$, then $\Omega$ is a direct tract of $u$.
% there exists $R>0$ such that $u$ is analytic in $\Omega_R = \{ z:  |z| > R \}$.
% {\color{red} and both of $\Omega$ and $\C \backslash \Omega$ are unbounded, which implies that $\Omega$ is a direct tract of $u$}.
Apply Theorem \ref{Wiman-Valiron-mero1} to the equation \eqref{third-order-equation} with $z = z_r$,
% there exists $F \subset [1, \infty) $ of finite logarithmic measure such that for $r \in [r_0, \infty)\backslash F$
then, as $r \rightarrow \infty, r \not \in F$, we obtain
\begin{align}
&A \left( \dfrac{a(r)}{z_r} \right)^2   u(z_r)^2 (1 + o(1))+B u(z_r) \left( \dfrac{a(r)}{z_r} \right)^2   u(z_r) (1 + o(1))  \nonumber
\\
=& \left( \dfrac{a(r)}{z_r} \right)^3   u(z_r) (1 + o(1)) +\alpha  \left( \dfrac{a(r)}{z_r} \right)^2   u(z_r) (1 + o(1))+\beta  \left( \dfrac{a(r)}{z_r} \right)   u(z_r) (1 + o(1)) \label{step10}
\\
&+\gamma  u(z_r) + \delta. \nonumber
\end{align}
Take the modulus on both sides of equation \eqref{step10},  we then have
\begin{eqnarray*}
%\label{third-order-ode}
&&\left|(A+B) \left( \dfrac{a(r)}{z_r} \right)^2    \right| M^2(r)  (1 + o(1))  \\
&=& \left| \left( \dfrac{a(r)}{z_r} \right)^3   u(z_r) (1 + o(1)) +\alpha  \left( \dfrac{a(r)}{z_r} \right)^2   u(z_r) (1 + o(1)) \right.
\\
&&\left. +\beta  \left( \dfrac{a(r)}{z_r} \right)   u(z_r) (1 + o(1)) \label{step1-0}
+\gamma  u(z_r) + \delta \right| %\nonumber
\\
&\leq& a^3(r)\left| u(z_r) \right| + |\alpha| a^2(r)\left| u(z_r) \right| +|\beta|a(r)\left| u(z_r) \right| \\
&&+|\gamma|\left| u(z_r) \right| + \left|\delta \right|
\\
&\leq&  K' a^3(r)     M(r)
\end{eqnarray*}
for some $K'>0$ and sufficiently large $r \not \in F$. After dividing by $M(r)$, we get
\begin{eqnarray*}
\left|(A+B) \left( \dfrac{a(r)}{z_r} \right)^2    \right| M(r)  (1 + o(1))
 \leq K' a^3(r).%\left(\left| \dfrac{}{z_r} \right|^3  +1  \right)  % = o (M(r)),
\end{eqnarray*}
%where the last equality   is deduced from Lemma \ref{Wiman-Valiron-mero2}. As a consequence,
According to Lemma \ref{Wiman-Valiron-mero2}, for any $\varepsilon \in (0, 1)$, when $r \not \in F$ is sufficiently large, the inequality
\begin{eqnarray*}
\left|A+B  \right| \left( \dfrac{a(r)}{r} \right)^2   M(r) (1 + o(1))   \leq M^{\varepsilon}(r)
\end{eqnarray*}
holds, and further by \eqref{Fuchs's result},  we have $M(r) > r^{\frac{2} {1-\varepsilon}}$ for all large $r$ and it follows from $a(r) \rightarrow \infty$ as $r \not \in \Lambda$ that
\begin{eqnarray*}
A+B = 0.
\end{eqnarray*}
Thus, we get a contradictation.
\end{proof}

\begin{corollary} \label{entire solution}
If $A+B \not = 0$, then the ODE \eqref{third-order-equation}
has no transcendental entire solutions.
\end{corollary}

\noindent {\it Proof of   Theorem \ref{Main Theorem}}.
Let $u$ be a meromorphic solution of \eqref{third-order-equation}.
If $u$ has no poles in $\C$, then according to Corollary \ref{entire solution}, $u$ must be a polynomial.
%$u = P$ is a  polynomial.
  % on $\C$.
Because of condition ii), by substituting $u$ into the equation \eqref{third-order-equation} and
 comparing the coefficients of the top degree terms, we conclude that the degree
of $P$ is at most two.

From now on, we consider the solution $u$ with at least one pole in $\C$.
Assume $z_0$ is a pole of $u$ with the Laurent expansion
\begin{eqnarray*}
u = \sum_{n=0}^{+\infty} u_n(z-z_0)^{n+p}, \quad u_0 \not = 0, - p\in \N.
\end{eqnarray*}
Because of condition ii), by substituting the above Laurent expansion
 into equation \eqref{third-order-equation}, we find that
\begin{itemize}

 \item [1)]  $p = -1 $,  $ u_0 = -\dfrac{6}{A+2 B}$;

\item [2)] the Fuchs indices (see Appendix A) of the equation \eqref{third-order-equation}  are
\begin{eqnarray*}
j= -1, j_2,j_3,
\end{eqnarray*}
where $j_{2,3} = \dfrac{7 A+8 B \pm \sqrt{25 A^2+16 A B-32 B^2}}{2 (A+2 B)}$.
\end{itemize}
Therefore, because of condition iii), all the $u_n, n\geq 1  $,  are uniquely determined by $u_0$.
This implies the existence of only one formal %RC NOT formal, CONVERGENT (somewhere in Gambier or Chazy)
Laurent
expansion with a pole at $z_0$ that satisfies the equation  \eqref{third-order-equation} (see Appendix A).

Now we distinguish two cases.
\begin{itemize}

\item [{\it Case 1.}] $u$ is  a nonconstant rational function. %We have two subcases.
%\begin{itemize}
%
%
% \item [Subcase (1a).]  $u = P$ is a  polynomial.
%
%  % on $\C$.
% As $B \not =  nA/(1-n) , n \in \Z \backslash \{0\}$ (condition ii)),   by substituting $u$ into to the equation \eqref{third-order-ode} and
% comparing the coefficient of the top degree, we can conclude that the degree
%of $P$ is at most 2.
%
% \item [Subcase (1b).]   $u$ has  poles on $\C$.

  If $u$ has at least two poles, say at $z_1$ and $z_2 (z_1 \not = z_2)$,
	then both of $u_1 = u(z-z_0+z_1)$ and $u_2=u(z-z_0+z_2)$ are solutions of \eqref{third-order-equation} with a pole at $z_0$.
	As there is a unique Laurent expansion around $z_0$,
	we must have $u_1 \equiv u_2$ on $\C$,
	which implies $\forall z \in \C, \ u(z) = u(z + z_2 - z_1)$,
	and  $u$ is periodic.
	This is  a contradiction because a nonconstant rational function cannot be periodic.

  If $u$ has only one pole and condition ii) is enforced, then
 \begin{eqnarray} \label{step1-2}
u(z) = \dfrac{u_0}{z-z_0} + P(z), %\quad u_0 \not = 0, - p\in \N.
\end{eqnarray}
 where $P(z)$ is a polynomial.
To further prove $\deg(P)\leq 2$, one may substitute  \eqref{step1-2} into the equation \eqref{third-order-equation}
and consider  the coefficient of the term with top degree in $z$.

 \item [{\it Case 2.}] $u$ is  a transcendental meromorphic function. %We have two subcases.

 According to Lemma \ref{infinitely many poles}, $u$ has infinitely many poles on $\C$. With the same argument as that in {\it Case 1}, we    conclude that $u$ is periodic.
As the set of all poles of $u$ is $\{z_0 + w | w \in \Gamma\}$,
where $\Gamma$ is a non-trivial discrete subgroup of $(\C, +)$ \cite[p.~57]{Jones1987Singerman},
we conclude that $u$ is either a doubly periodic (elliptic) function or a simply periodic function.

If $u$ is elliptic, then using the same argument as in {\it Case 1} it must have only one pole in the fundamental parallelogram and this pole is known to be simple.
This is impossible because the sum of residues of all poles inside the fundamental parallelogram of an elliptic function is zero.

Hence $u$ is simply periodic, and again it must have only one pole in the period stripe and this pole is  simple.
Then $u$ can be represented as $h(e^{k z})$, where $k \in \C^*$ and $h$ is a meromorphic function on $\C^* $ which has only one simple pole on $\C^*$.

Let us now prove by contradiction that $h$ is rational.
  Suppose $h$ has an essential singularity at infinity.
If we let $\zeta = e^{k z}$, the equation \eqref{third-order-equation} becomes
 \begin{eqnarray}
\label{transformed equation}
A \zeta ^2 h'^2+B h \left( \zeta ^2 h''+ \zeta  h' \right) = k \zeta ^3 h^{(3)}+   3k \zeta ^2h'' + \dfrac{(k^2   +\beta )}{k}\zeta  h' + \dfrac{\gamma}{k^2} h+ \dfrac{\delta}{k^2},
 \end{eqnarray}
  where $\alpha$ is set to 0 by a translation in $u$ (or $h$). As $h$ has   only one pole in $\C^*$, then there exists $L>0$ such that $h$ is analytic in $D = \{ \zeta \in \C:  |\zeta| \geq L \}$.
 Let $R = \max_{|\zeta| = L}|h(\zeta)|$. Since $h$ has  an essential singularity at infinity, by the big Picard theorem, there exists some $\xi_0$ with $|\xi_0|>L$ such that $|h(\xi_0)| >R$, so the set $S = \{\xi \in \C:|h(\xi)|>R, |\xi|>L\}$ is non-empty.  Let $\Omega$ be a component of the set $S$.  Since $h$ has an essential singularity at infinity, $\C \backslash \Omega$ must be unbounded. We claim that $\Omega$ is also unbounded and hence $\Omega$ is a direct tract of $h$.  If $\Omega$ is bounded, then $|h(\xi)|=R$ on $\partial\Omega$ while $|h(\xi)|>R$ in $\Omega$ which contradicts the maximum modulus principle.
Apply Theorem \ref{Wiman-Valiron-mero1} to the equation \eqref{transformed equation} with $\zeta = \zeta_r$,
% there exists $F \subset [1, \infty) $ of finite logarithmic measure such that for $r \in [r_0, \infty)\backslash F$
then, as $r \rightarrow \infty, r \not \in F$, we obtain
\begin{eqnarray*}
&&A\zeta_r^2 \left( \dfrac{a(r)}{\zeta_r} \right)^2  h(\zeta_r)^2 (1 + o(1))
\\
&&+ B h(\zeta_r) \left[ \zeta_r^2 \left( \dfrac{a(r)}{\zeta_r} \right)^2   h(\zeta_r)  +  \zeta_r \left( \dfrac{a(r)}{\zeta_r} \right)   h(\zeta_r) \right](1 + o(1))\\
&=& k \zeta_r^3 \left( \dfrac{a(r)}{\zeta_r} \right)^3   h(\zeta_r) (1 + o(1)) +  3 k  \zeta_r^2  \left( \dfrac{a(r)}{\zeta_r} \right)^2   h(\zeta_r) (1 + o(1))+
\\
&& \dfrac{(k^2  +\beta )}{k} \zeta_r    \left( \dfrac{a(r)}{\zeta_r} \right)   h(\zeta_r) (1 + o(1)) +\dfrac{\gamma}{k^2}  h(\zeta_r) + \dfrac{\delta}{k^2}.
\end{eqnarray*}
%Choose $K' > \max_{|\zeta| = R}|h(\zeta)|$. Let $\Omega$ be a component of the set $\{\zeta: |h(\zeta)| > K', \zeta > R\}$ and $M_h(r)= \max \{|h(\zeta)|:|\zeta|=r, \zeta \in \Omega \}$.
With similar argument as  in the proof of Lemma \ref{infinitely many poles},  for any $\varepsilon \in (0, 1)$, when $r   \not \in F$ is sufficiently large, the inequality
\begin{eqnarray*}
\left|A+B    \right| a^2(r)    M_h(r)  (1 + o(1)) \leq M_h^{\varepsilon}(r)
\end{eqnarray*}
holds, where $M_h(r)= \max \{|h(\zeta)|:|\zeta|=r, \zeta \in \Omega \}$, and further by \eqref{Fuchs's result},  we have
\begin{eqnarray*}
A + B = 0,
\end{eqnarray*}
and thus, we get a contradiction.

 When $h$ has an essential singularity at $\zeta = 0$, we may let $\eta = e^{-k z}$ and consider the function $g(\eta) = h(1/\eta)$. Then $g$ is a meromorphic function on $\C^* $ with an essential singularity at infinity and again similar arguments as above lead to a contradiction. As a consequence, neither $0$ nor infinity is an essential singularity of $h$, and hence $h$   is a rational function of the form
 \begin{eqnarray*}
h(\zeta) = \dfrac{h_0}{\zeta-\zeta_0} + P_1(\zeta) + P_2 (1/\zeta), \quad h_0, \zeta_0 \in \C^*, % \not = 0, %\quad u_0 \not = 0, - p\in \N.
\end{eqnarray*}
where $P_1 $ and $ P_2$ are polynomials, and $u$ is expressed as
 \begin{eqnarray*}
u(z) = h(e^{kz})=  \dfrac{h_0}{e^{kz}-\zeta_0} + P_1(e^{kz}) + P_2 (e^{-kz}), \quad k, h_0, \zeta_0 \in \C^*. % \not = 0, %\quad u_0 \not = 0, - p\in \N.
\end{eqnarray*}

Denote by \begin{eqnarray*}
            P_1(w) &=& c_{n} w^{n}  +\cdots +  c_{1} w  +c_0,\\
            P_2(w) &=& d_{m} w^{m}  +\cdots +  d_{1} w  +d_0,
          \end{eqnarray*}
          where  $c_{n}, d_{m} \in \C^*$ and $ n, m \in \N \cup \{0\}$ are the degrees of $P_1 $ and $P_2 $ respectively.

Assume $m \geq 1$, then as $\zeta \rightarrow 0,$ from \eqref{transformed equation}, we have
\begin{eqnarray*}
    && A \zeta^2 \left( -m \dfrac{d_m}{\zeta^{m+1}} \right)^2 +B  \dfrac{d_m}{\zeta^{m}} \left[ \zeta^2 m (m+1) \dfrac{d_m}{\zeta^{m+2}}  + \zeta(-m) \dfrac{d_m}{\zeta^{m+1}} \right] +O(\zeta^{-2m+1}) =  O(\zeta^{-m})
\end{eqnarray*}
and it can be simplified as
\begin{eqnarray*}
  m^2 d_m^2 (A+B) \zeta^{-2m} + O(\zeta^{-2m+1}) = O(\zeta^{-m})
\end{eqnarray*}
which contradicts to $A+B \not = 0$. Hence, $m = 0$ and $P_2$ is a constant.

Suppose $n\geq 2$, then   when $\zeta $ tends to infinity, from \eqref{transformed equation}, we have
\begin{eqnarray*}
    && A \zeta^2 \left( n  c_n \zeta^{n-1} \right)^2 +B  c_n\zeta^{n} \left[ \zeta^2 n (n-1)  c_n \zeta^{n-2}   + \zeta n c_n \zeta^{n-1} \right] +O(\zeta^{2n - 1})=  O(\zeta^{n})
\end{eqnarray*}
which implies that
\begin{eqnarray*}
 n^2 c_n^2 (A+B) \zeta^{2n} + O(\zeta^{2n-1}) = O(\zeta^{n}).
\end{eqnarray*}
Again, it contradicts to $A+B \not = 0$. Similarly, one may also show that $n \not = 1$, so we conclude that $P_1$ is also a constant and $u$ is expressed as
\begin{eqnarray*}
  u(z) &=& \dfrac{h_0}{e^{k z} - \zeta_0} + c_0, \quad  h_0, k, \zeta_0 \in \C^*, c_0 \in \C.
\end{eqnarray*}
This completes the proof of Theorem \ref{Main Theorem}.

 \end{itemize}

% {\color{red}

\section{Conclusion} \label{conclusion}

In this paper, we studied the third order ODE \eqref{third-order-equation} which includes the   Falkner-Skan equation and the \Cheng\ equation as special cases. For the generic case (see Definition \ref{generic cases}), all meromorphic solutions of this equation were derived by using complex analytic methods. Certain non-generic cases of equation \eqref{third-order-equation} with the Painlev\'{e} property were displayed as well. Our results show that, to find other closed-form solutions of equation \eqref{third-order-equation} in the future, one can either study meromorphic solutions of the remaining non-generic cases or focus on non-meromorphic solutions which are expected to have more complicated singularities than poles.  %Finally, the physical implications of our results await future efforts of researchers.

\section*{Declaration of Competing Interest}

The authors declare that there is no competing interest.

\section*{Acknowledgments}

%{\color{red} We would like to thank the referees for their very valuable comments.}
The first and second authors were partially supported by PROCORE - France/Hong Kong joint research grant, F-HK39/11T
 and RGC grant 17301115. The first and third authors were partially  supported by the National Natural Science Foundation of China (grant nos. 11701382 and 11971288).
The third author was partially supported by Guangdong Basic and Applied Basic Research Foundation, China (grant no. 2021A1515010054).

 \section*{Appendix A}
In this appendix, we give a brief summary on Painlev\'{e} test.

 Let  $I = (i_0, i_1, \dots, i_n), i_k \in \N \cup \{ 0\}, 0 \leq k \leq n$ and
% let $f = \chi^p$, then we have
 \begin{equation*}
H(y,y',\cdots,y^{(n)}) = \sum_{I \in \Lambda  } c_I y^{i_0} (y')^{i_1} \cdots (y^{(n)})^{i_n}, y = y(z), c_I  \in \C \backslash \{0\}.
\end{equation*}
If $y = (z - z_0)^p , -p \in \N$, then
 \begin{equation*}
H(y,y',\cdots,y^{(n)}) =  \sum_{I \in \Lambda } C_I (z - z_0)^{\alpha_I},
\end{equation*}
where $C_I \in \C, \alpha_I =  i_0 p + i_1 (p - 1) + \cdots + i_n (p - n)$.

Next, let $A$ be the set of those negative integers $p$
such that $\min_{I \in \Lambda} \alpha_I$ is attained by at least two $I$'s.  For each $p \in A$, denote by $\Lambda' = \{I' \in \Lambda | \alpha_{I'} = \min_{I \in \Lambda} \alpha_I
\}$ and then we define the {\bf dominant terms} for each $p \in A$ to be
% possible pair $(p, q)$
 \begin{equation*}
\hat{E} = \sum_{I \in \Lambda'  } c_I y^{i_0} (y')^{i_1} \cdots (y^{(n)})^{i_n} .
\end{equation*}

Suppose $u(z) = \sum_{n=0}^{+\infty} u_n(z-z_0)^{n+p}(u_0 \not = 0, - p\in \N)$ with a pole at $z= z_0$ is a meromorphic solution of
\begin{equation} \label{algebraic ODE}
H(y,y',\cdots,y^{(n)})=0.
\end{equation}
%where $ H $ is a polynomial in $y$ and its derivatives with constant coefficients.
Then if we  plug $y = u(z)$ into $H$, we will get an expression of the form $E = \sum_{j = 0}^{+\infty}E_j \chi^{j+q} = 0$,
  where $\chi = z-z_0, E_j \in \C$.
Since $y = u(z)$ is a solution of $H = 0$, we must have $E_j = 0$, for all $j \in \N$.

On the other hand, for $j = 1, 2, \dots$,
we can express $E_
j  $ as:
\begin{equation} \label{recursion-painleve test}
E_j \equiv P(u_0; j)u_j + Q_j(\{u_l | l < j\})  ,
\end{equation}
where $P(u_0; j)$ is a polynomial in $j$ determined by $u_0$ and $Q_j$ is a polynomial in $j$ with coefficients in $u_l ( l < j)$.
In fact, it is known that \cite{Darboux1883} (see also  \cite[p.~15]{Conte1999})
\begin{equation} \label{indicial equation}
P(u_0; j) = \lim_{\chi \rightarrow 0} \chi^{-j-q} \hat{E}'( u_0\chi^{p})\chi^{j+p},
\end{equation}
where $\hat{E}'(u)$ is defined by
\begin{equation}
\hat{E}'(u) v :=  \lim_{\lambda \rightarrow 0} \dfrac{ \hat{E}(u + \lambda v) - \hat{E}(u) }{\lambda}.
\end{equation}
%For each $j$, the above equation is linear in $u_j$.
%For the equation \eqref{recursion-painleve test} to vanish identically,
In order to have $E_j = 0$ for all $j \in \N$,
we must have for each $j$, either
\begin{itemize}
\item [1)] $u_j$ is uniquely determined by $P(u_0; j)$ and $Q_j$, {\it or}

\item [2)] both  $P(u_0; j)$ and $Q_j$  vanish,

\end{itemize}
otherwise there is no meromorphic function satisfying $H(y,y',\cdots,y^{(n)}) = 0$.

Therefore if the polynomial $P(u_0; j)$ in $j$ does not have any nonnegative integer root, then each $u_j$ is uniquely determined by $P(u_0; j)$ and $Q_j$.

 \begin{definition}
 The zeros of $P(u_0; j)$ are defined to be the {\bf Fuchs indices} of the equation $H(y,y',\cdots,y^{(n)}) = 0$ and the {\bf indicial equation} is defined as  $P(u_0; j) = 0$.
\end{definition}

 \begin{definition}
The ODE \eqref{algebraic ODE}  is said to pass the   Painlev\'{e} test  if all its Fuchs indices are distinct integers, and  at every positive integer Fuchs index $j$, the condition $Q_{j}=0
$ is obeyed.

 % its    indicial equation $P(u_0; j) = 0$ has $n$ distinct   integer roots  $-1, j_1,j_2, \dots, j_{n-1}$ and   $Q_{j_k} = 0, \; k = 1,2,\dots, n-1$. % for certain $u_0 \not = 0$.
\end{definition}
%\end{theorem}

{\it E-mail address:} robert.conte@cea.fr

{\it E-mail address:} ntw@maths.hku.hk

{\it E-mail address:} cfwu@szu.edu.cn

\end{document}